\documentclass[reqno,letterpaper]{amsart}
\usepackage[vmargin=33truemm,hmargin=31truemm]{geometry}
\usepackage[svgnames]{xcolor}
\usepackage{gnuplot-lua-tikz}
\usepackage{amssymb}
\usepackage{mathtools}
\mathtoolsset{showonlyrefs, showmanualtags}
\usepackage{bm}
\usepackage{mleftright}
\usepackage{aligned-overset}
\usepackage[bottom]{footmisc}
\usepackage{comment}

\numberwithin{equation}{section}

\theoremstyle{plain}
\newtheorem{theorem}{Theorem}[section]
\newtheorem{corollary}[theorem]{Corollary}
\newtheorem{proposition}[theorem]{Proposition}

\newcommand{\R}{\mathbb{R}}
\newcommand{\C}{\mathbb{C}}
\newcommand{\Z}{\mathbb{Z}}

\newcommand{\abs}[1]{\lvert#1\rvert} 
\newcommand{\norm}[1]{\lVert#1\rVert} 

\newcommand{\Hamiltonian}{H}
\newcommand{\Hamiltonianfree}{\Hamiltonian_0}
\newcommand{\HamiltonianDirac}[1]{\Hamiltonian_{\textrm{D}, #1}}
\newcommand{\Laplacian}{\Delta}

\newcommand{\Tnu}[1]{T_{#1}}
\newcommand{\TnumDirac}[2]{\widetilde{T}_{#1, #2}}

\DeclareMathOperator*{\esssup}{ess\,sup}

\newcommand{\Const}{\bm{C}}
\newcommand{\Constd}[1]{\Const^{(#1)}}
\newcommand{\Constwpd}[3]{\Const^{(#3)}(#1, #2)}
\newcommand{\ConstwpdDirac}[4]{\Const_{\textrm{D}, #4}^{(#3)}(#1, #2)}

\newcommand{\BesselJ}[1]{J_{#1}}
\newcommand{\BesselK}[1]{K_{#1}}
\newcommand{\BesselY}[1]{Y_{#1}}

\newcommand{\BesselJY}[1]{\mathbf{J}_{#1}}
\newcommand{\crossmunu}[2]{X_{#1, #2}}

\newcommand{\intervaloo}[2]{(#1, #2)}
\newcommand{\intervalco}[2]{[#1, #2)}

\newcommand{\intervalcc}[2]{[#1, #2]}

\newcommand{\widthL}{1}
\newcommand{\dimN}{N}

\newcommand{\citeDLMF}[1]{\cite[\href{https://dlmf.nist.gov/#1}{#1}]{DLMF}}

\usepackage[numbers,sort]{natbib}

\usepackage{hyperref}
\hypersetup{
	setpagesize=false,
	bookmarksnumbered=true,
	bookmarksopen=true,
	colorlinks=true,
	allcolors=blue,
}

\title{An order-interpolation inequality for Bessel functions}
\date{}

\author{Soichiro Suzuki}
\address[Soichiro Suzuki]{Department of Mathematics, Chuo University, 1-13-27, Kasuga, Bunkyo-ku, Tokyo 112-8551, Japan}
\email{\href{mailto:soichiro.suzuki.m18020a@gmail.com}{soichiro.suzuki.m18020a@gmail.com}}



\begin{document}
	\begin{abstract}
		We show that $J_{\mu + \nu}(r)^2 < J_{\nu-1/2}(r)^2 + J_{\nu+1/2}(r)^2$ holds whenever $\mu \in (-1/2, 1/2)$, $\nu \in [0, \infty)$, and $r \in (0, \infty)$. 
		In fact, we prove a stronger version for any fixed non-trivial linear combination of
		the Bessel functions of the first and second kinds.
		This inequality can be regarded as a kind of interpolation with respect to order. 
		As an application, we establish a dimension-comparison result for optimal constants of smoothing estimates for the free Schr\"{o}dinger equation.
		Briefly, the optimal constant on $\mathbb{R}^{d+1}$ is at most twice that on $\mathbb{R}^d$ for each $d \geq 2$.
	\end{abstract}
	\maketitle
	\section{Introduction}
	It is classical and well known that the Bessel function of the first kind $\BesselJ{\nu}$ satisfies
	\begin{equation} \label{eq:classical bound}
		r \BesselJ{\nu}(r)^2 \leq \frac{2}{\pi}
	\end{equation}
	for every $\nu \in \intervalcc{-1/2}{1/2}$ and $r \in \intervaloo{0}{\infty}$, which was proved by \citet[Equation (18)]{Sze1933} (see also \citet[Theorem 7.31.2]{Sze1975}).
	Comparing this with Hankel's asymptotic formula \citeDLMF{10.17.3}
	\begin{equation} \label{eq:Hankel's asymptotic formula}
	r \BesselJ{\nu}(r)^2 = \frac{1}{\pi} \mleft( 1 + \sin (2r - \nu \pi) \mright) + O(r^{-1}) , \quad r \to \infty
	\end{equation}
	for each fixed $\nu \in \R$, we see that the constant $2/\pi$ in \eqref{eq:classical bound} is sharp.
	On the other hand, \citet{Sze1933} also showed that \eqref{eq:classical bound} fails for $\nu \in \R \setminus \intervalcc{-1/2}{1/2}$, even though the asymptotic formula \eqref{eq:Hankel's asymptotic formula} is valid for every $\nu \in \R$.
	Indeed, we have
	\begin{alignat}{3} 
		&\nu \in \intervaloo{1/2}{\infty} \cup \Z_{\leq -1} & &\implies& \frac{2}{\pi} < &\sup_{r \in \intervaloo{0}{\infty}} r \BesselJ{\nu}(r)^2 < \infty , 
		\label{eq:classical bound nu > 1/2} \\
		&\nu \in \intervaloo{-\infty}{-1/2} \setminus \Z_{\leq -1} & &\implies&  &\sup_{r \in \intervaloo{0}{\infty}} r \BesselJ{\nu}(r)^2 = \infty . 
		\label{eq:classical bound nu < -1/2}
	\end{alignat}
	Moreover, \citet[Section 3.2]{Lan2000} proved that the function
	\begin{equation}
		\intervalco{1/2}{\infty} \ni \nu \longmapsto \sup_{r \in \intervaloo{0}{\infty}} r \BesselJ{\nu}(r)^2
	\end{equation}
	is strictly increasing and diverges to infinity as $\nu \to \infty$.
	\citet[Theorem 3]{Kra2014} showed that
	\begin{equation} \label{eq:Kra2014 bound}
		\abs{r^2 - \nu^2 + 1/4}^{1/2} \BesselJ{\nu}(r)^2 < \frac{2}{\pi}
	\end{equation}
	holds for every $\nu \in \intervaloo{1/2}{\infty}$ and $r \in \intervaloo{0}{\infty}$, which is a natural replacement of \eqref{eq:classical bound}.
	The purpose of this paper is to give another generalization in terms of an interpolation inequality with respect to order.
	Notice that the inequality \eqref{eq:classical bound} can be rewritten as
	\begin{equation} \label{eq:classical bound rewritten}
		\BesselJ{\nu}(r)^2 \leq \BesselJ{-1/2}(r)^2 + \BesselJ{1/2}(r)^2 
	\end{equation}
	by using 
	\begin{equation} \label{eq:BesselJ for nu = pm 1/2}
		\BesselJ{-1/2}(r) = \mleft( \frac{2}{\pi r} \mright)^{1/2} \cos{r} , \quad \BesselJ{1/2}(r) = \mleft( \frac{2}{\pi r} \mright)^{1/2} \sin{r} .
	\end{equation}
	This observation leads us to the following.
	\begin{theorem} \label{thm:main thm}
		We define a family of functions $\{ \BesselJY{\nu} \}_{\nu \in \R}$ by
		\begin{equation} \label{eq:BesselJY}
			\BesselJY{\nu} \coloneqq a \BesselJ{\nu} + b \BesselY{\nu}
		\end{equation}
		for fixed $(a, b) \in \C^2 \setminus \{(0, 0)\}$, where $\BesselJ{\nu}$ and $\BesselY{\nu}$ are the Bessel functions of the first and second kinds of order $\nu$, respectively.
		Now let $\mu \in \intervaloo{-\widthL/2}{\widthL/2}$, $\nu \in \intervalco{0}{\infty}$, and $r \in \intervaloo{0}{\infty}$.
		Then we have
		\begin{equation} \label{eq:main inequality}
			\abs{\BesselJY{\mu + \nu}(r)}^2 < \abs{\BesselJY{\nu-\widthL/2}(r)}^2 + \abs{\BesselJY{\nu+\widthL/2}(r)}^2 .
		\end{equation}
	\end{theorem}
	To the best of our knowledge, this simple-looking inequality does not seem to have appeared in the literature.
	See Section \ref{section:proof} for the proof.
	We remark that neither \eqref{eq:Kra2014 bound} nor \eqref{eq:main inequality} is stronger than the other.
	For example, \eqref{eq:Kra2014 bound} and \eqref{eq:main inequality} imply that
	\begin{equation} \label{eq:upper bounds for nu=1}
		\frac{1}{2} \pi r \BesselJ{1}(r)^2 
		< \begin{dcases}
			\frac{r}{\abs{r^2 - 3/4}^{1/2}} , \\
			\frac{1}{2} \pi r (\BesselJ{1/2}(r)^2 + \BesselJ{3/2}(r)^2) =  1 - \frac{\sin{2r}}{r} + \frac{1 - \cos{2r}}{2 r^2}  ,
		\end{dcases}
	\end{equation} 
	respectively. Then it is easy to see that 
	\begin{equation}
		1 - \frac{\sin{2r}}{r} + \frac{1 - \cos{2r}}{2 r^2}  < \frac{r}{\abs{r^2 - 3/4}^{1/2}}
	\end{equation}
	holds when $r = \pi/4 + n \pi$, while the reverse inequality holds when $r = 3\pi/4 + n \pi$, where $n \in \Z_{\geq 0}$. 
	Figure \ref{fig:graph} shows the graphs of the functions appearing in \eqref{eq:upper bounds for nu=1}. 
	\begin{figure}[b]
		\centering
		\input{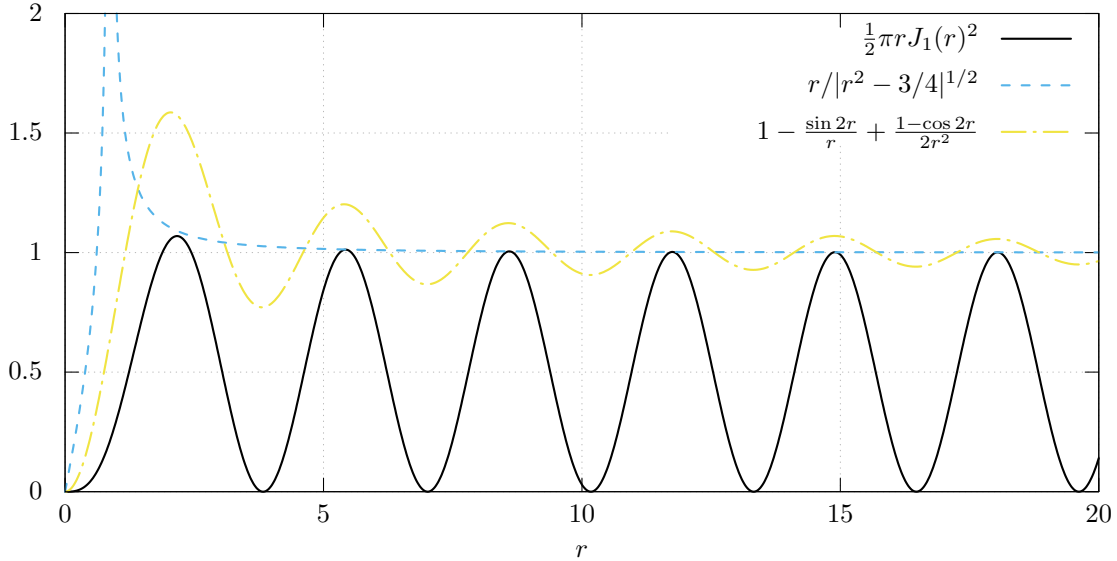}
		\caption{The graphs of the functions appearing in \eqref{eq:upper bounds for nu=1}.}
		\label{fig:graph}
	\end{figure}
	
	In addition, we establish a dimension-comparison result for the free Schr\"{o}dinger equation on $\R^d$ and $\R^{d+1}$ by using Theorem \ref{thm:main thm}. 
	Briefly, we have
	\begin{equation}
	\Constd{d+1} \leq 2 \Constd{d} 
	\end{equation}
	for every $d \geq 2$, where $\Constd{d}$ and $\Constd{d+1}$ denote the optimal constants of a certain inequality on $\R^d$ and $\R^{d+1}$, respectively.
	See Section \ref{section:application} for details.	
	\section{Proof of Theorem \ref{thm:main thm}} \label{section:proof}
	Our proof is based on three classical identities for the Bessel functions: the Wronskian identity, \citeauthor{Wat1944}'s formula \cite[Eq.~(5) in Section 13.73]{Wat1944}, and \citeauthor{Nic1910}'s formula \cite[Eqs.~(33), (34)]{Nic1910}.
	For each $(\mu, \nu) \in \R^2$, we define $\crossmunu{\mu}{\nu} \colon \intervaloo{0}{\infty} \to \R$ by
	\begin{equation} \label{eq:cross product}
		\crossmunu{\mu}{\nu}(r) \coloneqq \det \begin{pmatrix}
			\BesselJ{\mu} & \BesselJ{\nu} \\
			\BesselY{\mu} & \BesselY{\nu}
		\end{pmatrix}
		=
		\BesselJ{\mu}(r) \BesselY{\nu}(r) - \BesselJ{\nu}(r) \BesselY{\mu}(r) ,
	\end{equation}
	that is, a cross product of $\BesselJ{}$ and $\BesselY{}$.
	Then we have the following.
	\begin{proposition}[Wronskian identity] \label{prop:Wronskian}
		Let $\nu \in \R$ and $r \in \intervaloo{0}{\infty}$.
		Then we have
		\begin{equation} \label{eq:Wronskian}
			\crossmunu{\nu+1/2}{\nu-1/2}(r) = \frac{2}{\pi r} .
		\end{equation}
	\end{proposition}
	\begin{proposition}[\citeauthor{Wat1944}'s formula%
		\footnote{Some references (e.g.~\citet[6.617.1]{GR2014}) present \citeauthor{Wat1944}'s formula \eqref{eq:Watson} with the wrong sign in the exponential.}%
] \label{prop:Watson}
		Let $\mu \in \intervaloo{-1}{1}$, $\nu \in \R$, and $r \in \intervaloo{0}{\infty}$.
		Then we have
		\begin{equation} \label{eq:Watson}
			\crossmunu{\mu+\nu}{\nu}(r) = \frac{4 \sin( \mu \pi )}{\pi^2} \int_{s \in \intervaloo{0}{\infty}} \BesselK{\mu}(2 r \sinh{s}) \exp( - (\mu + 2\nu) s ) \, ds , 
		\end{equation}
		where $\BesselK{\mu}$ is the modified Bessel function of the second kind of order $\mu$.
	\end{proposition}
	\begin{proposition}[\citeauthor{Nic1910}'s formula]
		Let $\nu \in \R$ and $r \in \intervaloo{0}{\infty}$.
		Then we have
		\begin{equation} \label{eq:Nicholson}
			\BesselJ{\nu}(r)^2 + \BesselY{\nu}(r)^2 = \frac{8}{\pi^2} \int_{s \in \intervaloo{0}{\infty}} \BesselK{0}(2 r \sinh{s}) \cosh( 2 \nu s ) \, ds .
		\end{equation}
	\end{proposition}
	Since $\BesselK{\nu}$ is positive on $\intervaloo{0}{\infty}$ for every $\nu \in \R$, 
	we see that the identities \eqref{eq:Watson} and \eqref{eq:Nicholson} imply the following, respectively.
	\begin{corollary} \label{cor:Watson}
		Let $\mu \in \intervaloo{0}{1}$ and $r \in \intervaloo{0}{\infty}$. Then the function
		\begin{equation}
			\nu \longmapsto \crossmunu{\mu+\nu}{\nu}(r)
		\end{equation}
		is positive and strictly decreasing on $\R$.
	\end{corollary}
	\begin{corollary} \label{cor:Nicholson}
		Let $r \in \intervaloo{0}{\infty}$.
		Then the function
		\begin{equation}
			\nu \longmapsto \BesselJ{\nu}(r)^2 + \BesselY{\nu}(r)^2 
		\end{equation}
		is even on $\R$ and strictly increasing on $\intervalco{0}{\infty}$.
	\end{corollary}
	Now we prove Theorem \ref{thm:main thm} by using Proposition \ref{prop:Wronskian} and Corollaries \ref{cor:Watson}, \ref{cor:Nicholson}.  
	\begin{proof}[Proof of Theorem \ref{thm:main thm}]
		Using the Wronskian identity \eqref{eq:Wronskian} and Cramer's rule, we obtain
		\begin{equation}
			\frac{2}{\pi r}
			\begin{pmatrix}
				\BesselJ{\mu + \nu}(r) \\
				\BesselY{\mu + \nu}(r)
			\end{pmatrix}
			= \crossmunu{\nu+1/2}{\mu + \nu}(r)
			\begin{pmatrix}
				\BesselJ{\nu-1/2}(r) \\
				\BesselY{\nu-1/2}(r)
			\end{pmatrix}
			+ \crossmunu{\mu + \nu}{\nu-1/2}(r)
			\begin{pmatrix}
				\BesselJ{\nu+1/2}(r) \\
				\BesselY{\nu+1/2}(r)
			\end{pmatrix} ,
		\end{equation} 
		so that
		\begin{equation}
			\frac{2}{\pi r} \BesselJY{\mu+\nu}(r) = \crossmunu{\nu+1/2}{\mu + \nu}(r) \BesselJY{\nu-1/2}(r) + \crossmunu{\mu + \nu}{\nu-1/2}(r) \BesselJY{\nu+1/2}(r) .
		\end{equation}
		Thus, applying the Cauchy--Schwarz inequality, we get
		\begin{equation}
		\mleft(\frac{2}{\pi r}\mright)^2 \abs{\BesselJY{\mu+\nu}(r)}^2 \leq (\crossmunu{\nu+1/2}{\mu + \nu}(r)^2 + \crossmunu{\mu + \nu}{\nu-1/2}(r)^2) ( \abs{\BesselJY{\nu-1/2}(r)}^2 + \abs{\BesselJY{\nu+1/2}(r)}^2 ) .
		\end{equation}
			We also note that $\crossmunu{\nu+1/2}{\nu-1/2}(r) \ne 0$ (which follows from the Wronskian identity \eqref{eq:Wronskian}) and the assumption $(a, b) \ne (0, 0)$ imply that 
		\begin{equation}
			\abs{\BesselJY{\nu-1/2}(r)}^2 + \abs{\BesselJY{\nu+1/2}(r)}^2 > 0 .
		\end{equation}
		Therefore, it suffices to show that 
		\begin{equation}
			\crossmunu{\nu+1/2}{\mu + \nu}(r)^2 + \crossmunu{\mu + \nu}{\nu-1/2}(r)^2 < \mleft(\frac{2}{\pi r}\mright)^2
		\end{equation}
		holds.
		Moreover, Corollary \ref{cor:Watson} gives us
		\begin{align}
			0 < \crossmunu{(1/2-\mu) + \mu + \nu}{\mu + \nu}(r) &\leq \crossmunu{(1/2-\mu) + \mu}{\mu}(r) , \\
			0 < \crossmunu{\mu + \nu}{\nu-1/2}(r) &\leq \crossmunu{\mu}{-1/2}(r) , 
		\end{align}
		so that it is reduced to
		\begin{equation} \label{eq:sufficient for L=1}
			\crossmunu{1/2}{\mu}(r)^2 + \crossmunu{\mu}{-1/2}(r)^2 < \mleft(\frac{2}{\pi r}\mright)^2 .
		\end{equation}
		Now observe that the left-hand side can be simplified as
		\begin{equation}
			\crossmunu{1/2}{\mu}(r)^2 + \crossmunu{\mu}{-1/2}(r)^2 = \frac{2}{\pi r} \mleft( \BesselJ{\mu}(r)^2 + \BesselY{\mu}(r)^2 \mright) 
		\end{equation}
		by substituting
		\begin{equation}
			\BesselJ{-1/2}(r) = - \BesselY{1/2}(r) = \mleft( \frac{2}{\pi r} \mright)^{1/2} \cos{r} , \quad \BesselJ{1/2}(r) = \BesselY{-1/2}(r) = \mleft( \frac{2}{\pi r} \mright)^{1/2} \sin{r} .
		\end{equation}
		Then we get
		\begin{equation}
			\BesselJ{\mu}(r)^2 + \BesselY{\mu}(r)^2 < \BesselJ{1/2}(r)^2 + \BesselY{1/2}(r)^2 = \frac{2}{\pi r} 
		\end{equation}
		by Corollary \ref{cor:Nicholson}, since $\mu \in \intervaloo{-1/2}{1/2}$.
		This completes the proof.
	\end{proof}
	\section{Application to optimal constants of smoothing estimates} \label{section:application}
	Let $\Hamiltonianfree \coloneqq - \Laplacian$ be the Schr\"{o}dinger operator of a free particle on $\R^d$. Then it is well known that the so-called smoothing estimate
\begin{equation} \label{eq:smoothing free}
	\int_{(x, t) \in \R^d \times \R} w(\abs{x}) \abs{ \psi(\Hamiltonianfree) e^{- i t \Hamiltonianfree} u_0(x) }^2 \, dx \, dt \leq C \norm{u_0}_{L^2(\R^d)}^{2} 
\end{equation}
holds for a suitable pair of functions $w, \psi \colon \intervaloo{0}{\infty} \to \intervalco{0}{\infty}$, for example,
\begin{alignat}{8}
	&d \geq 3,  \quad && \quad && ( w(r), \psi(s) ) = ( &&(1+r^2)^{-1}, \quad&&(1+s)^{1/4} &&) ,
	\label{eq:type A} 
\end{alignat}
which is due to \citet[Theorem 2]{KY1989}. The inequality \eqref{eq:smoothing free} is usually referred to as a smoothing estimate.
Hereinafter, we always assume that $w$ and $\psi$ are non-negative Borel measurable functions defined on $\intervaloo{0}{\infty}$.
Now let $\Constwpd{w}{\psi}{d} \in \intervalcc{0}{\infty}$ be the optimal constant of the inequality \eqref{eq:smoothing free}, 
that is,
\begin{equation}
	\Constwpd{w}{\psi}{d} \coloneqq \sup_{u_0 \in L^2(\R^d) \setminus \{0\}} \frac{1}{\norm{u_0}_{L^2(\R^d)}^{2}} \int_{(x, t) \in \R^d \times \R} w(\abs{x}) \abs{ \psi(\Hamiltonianfree) e^{- i t \Hamiltonianfree} u_0(x) }^2 \, dx \, dt .
\end{equation}
\citet[Theorem 4.1]{Wal2002} established the following formula for the optimal constant. 
\begin{theorem}[{\cite[Theorem 4.1]{Wal2002}}] \label{thm:Walther}
	Let $w \colon \intervaloo{0}{\infty} \to \intervalco{0}{\infty}$.
	For each $\nu \in \R$, we define $\Tnu{\nu} w \colon \intervaloo{0}{\infty} \to \intervalcc{0}{\infty}$ by
	\begin{equation} \label{eq:Tnu}
		\Tnu{\nu}w(s) \coloneqq \pi \int_{r \in \intervaloo{0}{\infty}} r s \BesselJ{\nu}(rs)^2 w(r) \, dr .
	\end{equation}
	Then, for every $d \geq 2$ and $\psi \colon \intervaloo{0}{\infty} \to \intervalco{0}{\infty}$, we have
	\begin{equation} \label{eq:Walther}
		\Constwpd{w}{\psi}{d} = \sup_{k \in \Z_{\geq 0}} \esssup_{s \in \intervaloo{0}{\infty}} s^{-1} \psi(s^2)^2 \Tnu{k+d/2-1}w(s) .
	\end{equation}
\end{theorem}
Notice that \citeauthor{Wal2002}'s formula immediately implies
\begin{equation} \label{eq:d+2 vs d}
	\Constwpd{w}{\psi}{d+2} \leq \Constwpd{w}{\psi}{d}
\end{equation}
for every $d \geq 2$ and $w, \psi \colon \intervaloo{0}{\infty} \to \intervalco{0}{\infty}$, since
\begin{equation}
	\Constwpd{w}{\psi}{d+2} = \sup_{k \in \Z_{\geq 0}} \esssup_{s \in \intervaloo{0}{\infty}} s^{-1} \psi(s^2)^2 \Tnu{k+(d+2)/2-1}w(s) = \sup_{k \in \Z_{\geq 1}} \esssup_{s \in \intervaloo{0}{\infty}} s^{-1} \psi(s^2)^2 \Tnu{k+d/2-1}w(s) .
\end{equation}
Therefore, it is natural to ask if we have
\begin{equation}
	\Constwpd{w}{\psi}{d+1} \lesssim \Constwpd{w}{\psi}{d} .
\end{equation}
Our Theorem \ref{thm:main thm} gives an affirmative answer with an explicit constant.
In fact, Theorem \ref{thm:main thm} implies
\begin{equation}
	\Tnu{\nu}w(s) \leq \Tnu{\nu-1/2}w(s) + \Tnu{\nu+1/2}w(s) \leq 2 \max\{ \Tnu{\nu-1/2}w(s) , \, \Tnu{\nu+1/2}w(s) \}
\end{equation}
for every $w \colon \intervaloo{0}{\infty} \to \intervalco{0}{\infty}$, $\nu \in \intervalco{0}{\infty}$, and $s \in \intervaloo{0}{\infty}$.
Combining this inequality with \citeauthor{Wal2002}'s formula \eqref{eq:Walther}, we obtain
\begin{equation} \label{eq:d+1 vs d with const 2}
	\Constwpd{w}{\psi}{d+1} \leq 2 \Constwpd{w}{\psi}{d} .
\end{equation}
A related refinement is obtained by comparing these constants with the optimal constant
for the Dirac smoothing estimate.
Let $\HamiltonianDirac{m}$ be the Dirac operator of a free particle with mass $m \in \intervalco{0}{\infty}$ on $\R^d$, 
and let $\ConstwpdDirac{w}{\psi}{d}{m} \in \intervalcc{0}{\infty}$ be the optimal constant of the inequality
\begin{equation} \label{eq:smoothing Dirac}
	\int_{(x, t) \in \R^d \times \R} w(\abs{x}) \abs{ \abs{\HamiltonianDirac{m}}^{-1/2} \psi(\Hamiltonianfree) e^{- i t \HamiltonianDirac{m}} u_0(x) }^2 \, dx \, dt \leq C \norm{u_0}_{L^2(\R^d, \C^\dimN)}^{2} ,
\end{equation}
where $\dimN \coloneqq 2^{\lfloor (d+1)/2 \rfloor}$.
\citet{Suz2025} showed the following analogue of Theorem \ref{thm:Walther}.
\begin{theorem}[{\cite[Theorem 1.5]{Suz2025}}]
	Let $m \in \intervalco{0}{\infty}$ and $w \colon \intervaloo{0}{\infty} \to \intervalco{0}{\infty}$.
	For each $\nu \in \R$, we define $\TnumDirac{\nu}{m} w \colon \intervaloo{0}{\infty} \to \intervalcc{0}{\infty}$ by
	\begin{equation} \label{eq:Tnu Dirac}
	\TnumDirac{\nu}{m}w(s) \coloneqq \Tnu{\nu}w(s) + \Tnu{\nu+1}w(s) + \frac{m}{(m^2+s^2)^{1/2}} \abs{ \Tnu{\nu}w(s) - \Tnu{\nu+1}w(s) } ,
	\end{equation}
	where $\Tnu{\nu} w$ is as in \eqref{eq:Tnu}. 
	Then, for every $d \geq 2$ and $\psi \colon \intervaloo{0}{\infty} \to \intervalco{0}{\infty}$, we have
	\begin{equation} \label{eq:Walther Dirac}
		\ConstwpdDirac{w}{\psi}{d}{m} = \sup_{k \in \Z_{\geq 0}} \esssup_{s \in \intervaloo{0}{\infty}} s^{-1} \psi(s^2)^2 \TnumDirac{k+d/2-1}{m}w(s) .
	\end{equation}
\end{theorem}
Combining \eqref{eq:main inequality} and \eqref{eq:Tnu Dirac}, we see that
\begin{equation}
\max\{ \Tnu{\nu-1/2}w(s),\ \Tnu{\nu}w(s) \} \leq \TnumDirac{\nu-1/2}{m}w(s) \leq 2 \max\{ \Tnu{\nu-1/2}w(s),\ \Tnu{\nu+1/2}w(s) \} 
\end{equation}
for every $w \colon \intervaloo{0}{\infty} \to \intervalco{0}{\infty}$, $m \in \intervalco{0}{\infty}$, $\nu \in \intervalco{0}{\infty}$, and $s \in \intervaloo{0}{\infty}$.
Hence, we obtain the following:
\begin{theorem} \label{thm:dimension-comparison}
	We have
	\begin{equation} \label{eq:d+1 vs d with const 2 Dirac}
	\max\{ \Constwpd{w}{\psi}{d}, \, \Constwpd{w}{\psi}{d+1} \} \leq \ConstwpdDirac{w}{\psi}{d}{m} \leq 2 \Constwpd{w}{\psi}{d} 
	\end{equation}
	for every $d \geq 2$, $m \in \intervalco{0}{\infty}$, and $w, \psi \colon \intervaloo{0}{\infty} \to \intervalco{0}{\infty}$.
\end{theorem}
We note that the inequality \eqref{eq:d+1 vs d with const 2 Dirac} is sharp for every $d \geq 2$ in the following sense: we have
\begin{gather}
\lim_{p \uparrow d} \frac{\ConstwpdDirac{ w_p }{ \psi_p }{d}{0}}{ \max\{ \Constwpd{ w_p }{ \psi_p }{d}, \, \Constwpd{ w_p }{ \psi_p }{d+1} \} } = 1 , \\
\lim_{p \downarrow 1} \frac{\ConstwpdDirac{ w_p }{ \psi_p }{d}{0}}{ \Constwpd{w_p}{\psi_p}{d} } = 2
\end{gather}
for
$
( w_p(r) , \psi_p(s) ) = ( r^{-p}, s^{(2-p)/4} )
$;
see \cite[Theorem 1.6]{BS2017} and \citet[Theorem 1.8]{Suz2025} for details.
Meanwhile, as of this writing, the sharpness of the inequality \eqref{eq:d+1 vs d with const 2} remains open.
Indeed, we do not know if there exists $(w, \psi)$ such that
\begin{equation}
	\Constwpd{w}{\psi}{d} < \Constwpd{w}{\psi}{d+1} \leq 2 \Constwpd{w}{\psi}{d} .
\end{equation}
We remark that the following sufficient condition for
\begin{equation} \label{eq:d+1 vs d with const 1}
	\Constwpd{w}{\psi}{d+1} \leq \Constwpd{w}{\psi}{d} 
\end{equation}
was recently obtained by \citet{Suz2025_2}.
\begin{theorem}[{\cite[Theorem 1.8]{Suz2025_2}}] \label{thm:CM}
	Let $w \colon \intervaloo{0}{\infty} \to \intervalco{0}{\infty}$ be such that
	\begin{equation}
		w(r) = \int_{a \in \intervaloo{0}{\infty}} \exp(- a r^2) \, d\lambda(a)
	\end{equation} 
	for some non-negative Borel measure $\lambda$ on $\intervaloo{0}{\infty}$. 
	Then the function
	\begin{equation}
		\intervalco{0}{\infty} \ni \nu \longmapsto \Tnu{\nu}w(s) \in \intervalcc{0}{\infty}
	\end{equation}
	is non-increasing for each $s \in \intervaloo{0}{\infty}$. 
	Consequently, the inequality \eqref{eq:d+1 vs d with const 1} holds for every $d \geq 2$ and $\psi \colon \intervaloo{0}{\infty} \to \intervalco{0}{\infty}$.
\end{theorem}
For example, $w(r) = (1+r^2)^{-p/2}$ satisfies the assumption of Theorem \ref{thm:CM} for every $p \in \intervaloo{0}{\infty}$, since
\begin{equation}
	(1+r^2)^{-p/2}  = \frac{1}{\Gamma(p/2)} \int_{a \in \intervaloo{0}{\infty}} \exp(-a r^2) a^{p/2-1} \exp(-a) \, da .
\end{equation}
	\section*{Acknowledgments}
The author's original motivation was to prove
	\begin{equation} 
\BesselJ{\nu}(r)^2 \leq \BesselJ{\nu-1/2}(r)^2 + \BesselJ{\nu+1/2}(r)^2 
	\end{equation}
	in order to establish Theorem \ref{thm:dimension-comparison}. 
The author would like to thank Timothy Budd for suggesting the more general inequality
	\begin{equation} 
\BesselJ{\mu + \nu}(r)^2 \leq \BesselJ{\nu-1/2}(r)^2 + \BesselJ{\nu+1/2}(r)^2 
	\end{equation}
	on MathOverflow\footnote{\url{https://mathoverflow.net/posts/comments/1292953}}.

\end{document}